\documentclass[a4paper,abstract=true]{scrartcl}
\usepackage[utf8]{inputenc}
\usepackage{amsfonts,amsthm,amssymb,amsmath}
\usepackage[inline]{enumitem}
\usepackage{color}
\usepackage{tikz}

\newtheorem{ctr}{}[section]
\newtheorem{theorem}[ctr]{Theorem}

\newtheorem{corollary}[ctr]{Corollary}

\theoremstyle{remark}
\newtheorem{remark}[ctr]{Remark}
\newtheorem{question}[ctr]{Question}

\newcommand{\Aut}{\operatorname{Aut}}

\DeclareMathOperator{\wcop}{wCop}
\DeclareMathOperator{\scop}{sCop}

\title{On weak cop numbers of transitive graphs}
\author{Florian Lehner}
\begin{document}

\maketitle
\begin{abstract}
The weak cop number of infinite graphs can be seen as a coarse-geometric analogue to the cop number of finite graphs. We show that every vertex transitive graph with at least one thick end has infinite weak cop number. It follows that every connected, vertex transitive graph has weak cop number $1$ or $\infty$, answering a question posed by Lee, Martínez-Pedroza, and Rodríguez-Quinche, and reiterated in recent preprints by Appenzeller and Klinge, and by Esperet, Gahlawat, and Giocanti.
\end{abstract}

\section{Introduction}

The Cops-and-Robber game is a two-player perfect information game played on a (usually finite) graph as follows.
Initially, the first player places $k$ playing pieces, the cops, on the vertices of a graph $G$. Then the second player places his playing piece, the robber, on a vertex. We refer to the first player as the cops, and the second player as the robber. Once initial positions are established, the two players take turns. On the cops' turn, each cop can either move to an adjacent vertex or stay at the current position. On the robber’s turn, the robber can either move to an adjacent vertex or stay where they are.  The cops win the game if at some point one of the cops is at the same vertex as the robber, in this case we say that the robber is caught.

Too many variants of this game have been studied in the literature to even attempt to list them all here; instead, we refer the reader to the monograph \cite{zbMATH05945963}, or the more recent \cite{zbMATH07558343} for an introductory overview. One of the main parameters of interest in each of these variants is the associated cop number of a graph $G$, that is, the minimum number $k$ of cops required for the cops to have a winning strategy on $G$.

This short note focuses on a coarse-geometric variant of the game which is played on infinite graphs and the associated cop number. This variant was recently introduced by Lee, Martínez-Pedroza, and Rodríguez-Quinche \cite{zbMATH07745147} and is defined as follows.
There are four non-negative integer parameters, chosen by the players: the cop speed $s_c$, the robber speed $s_r$, the capture radius $\rho$, and the robber reach $R$. Once the parameters are chosen, the game proceeds as follows. In the first round, the cops, and then the robber choose their respective starting vertices. In every subsequent round, the cop player may move each cop to any vertex at distance $\leq s_c$ from their current position, and then the robber player moves the robber along a path of length at most $s_r$ starting at their current position. We say that the robber wins the game, if they manage to visit the ball of radius $R$ centred at some fixed vertex $v$ infinitely many times while never being at distance $\leq \rho$ to any cop (including at interior vertices of the path along which the robber moves). The cops win the game if they prevent this from happening. 

In order to rule out trivialities such as setting the opposing player's speed to $0$, the cops choose $s_c$, and the robber chooses $s_r$. Moreover, $s_c$ must be chosen before $s_r$ because $s_c > s_r$ would trivially lead to a certain cop-win even for a single cop. Further note that it is never in the robber's interest to pick $\rho > 0$. Hence, in order to make $\rho$ non-trivial, it must be chosen by the cops. Finally $R$ must be chosen after $\rho$, and it must be chosen by the robber, because $R \leq \rho$ is a trivial cop win, and the choice of $v$ does not matter as long as $v$ is fixed before $R$ is chosen.

This leaves two possible variants:
\begin{itemize}
    \item in the \emph{weak game}, the cops choose $\rho$ before the robber picks $s_r$;
    \item in the \emph{strong game}, the cops choose $\rho$ after the robber picks $s_r$.
\end{itemize}
The \emph{weak cop number} $\wcop(G)$ and \emph{strong cop number} $\scop(G)$ of a graph $G$ denote the minimum number of cops required to guarantee a win in the weak and strong game, respectively. Note that $\wcop(G)$ and $\scop(G)$ can be infinite.

Both $\wcop$ and $\scop$  are known to be invariant under quasi-isometries \cite{zbMATH07745147}, and hence can be thought of as generalisations of the cop number to the realm of coarse graph theory recently introduced by Georgakopoulos and Papasoglu \cite{georgakopoulos2023graphminorsmetricspaces}. The two parameters were further studied in recent preprints \cite{appenzeller2025copsrobbershyperbolicvirtually,esperet2025coarsecopsrobbergraphs}, both of which independently prove that $\scop(G) = 1$ if and only if $G$ is Gromov-hyperbolic.

The following question is one of the main open problems concerning weak and strong cop numbers, first stated in \cite[Question L]{zbMATH07745147}, and reiterated for the special case of Cayley graphs in \cite[Question K]{appenzeller2025copsrobbershyperbolicvirtually} and \cite[Question 1]{esperet2025coarsecopsrobbergraphs}.

\begin{question}
    \label{qu:main}
    Is there a connected vertex transitive graph $G$ for which $\wcop(G) \notin \{1,\infty\}$ or $\scop(G) \notin \{1,\infty\}$?
\end{question}

We show that any vertex transitive graph $G$ containing a thick end has infinite weak cop number. As a consequence, any connected vertex transitive graph $G$ satisfies $\wcop(G) \in \{1,\infty\}$, and thus Question \ref{qu:main} has a negative answer for $\wcop$.

\section{Notation}

We denote the vertex set of a graph $G$ by $V(G)$ and the edge set by $E(G)$. All graphs are assumed to be \emph{locally finite}, that is, every vertex has only finitely many neighbours. 

A \emph{ray} in an infinite, locally finite graph is a one-sided infinite path. An \emph{end} of $G$ is an equivalence class of rays where two rays are equivalent if they are connected by an infinite number of disjoint paths. The \emph{degree} of an end $\omega$ is the supremum of the sizes of families of vertex disjoint rays in $\omega$. An end is called \emph{thin} if it has finite degree, and \emph{thick} otherwise.

The \emph{automorphism group} of a graph $G$, denoted by $\Aut G$, is the group of bijections of $V(G)$ which preserve both adjacency and non-adjacency. A graph is called vertex transitive, if its automorphism group acts transitively on the set of vertices, that is, for any pair $u,v$ of vertices there is some $g \in \Aut G$ such that $g u = v$.

Let $d$ denote the usual geodesic distance on $V(G)$, that is $d(u,v)$ is the length of a shortest $u$--$v$-path in $G$. The \emph{ball} of radius $r$ around a vertex $v$, is defined as $B(r,v) = \{u\in V(G)\mid d(u,v) \leq r\}$; the \emph{sphere} of radius $r$ around a vertex $v$, is defined as $S(r,v) = \{u\in V(G)\mid d(u,v) = r\}$. Note that in a vertex transitive graph, the size of a ball only depends on the radius, not on the centre, and hence we may unambiguously define $b(n) = |B(n,v)|$ for every $n \in \mathbb N$.

\section{Results}

\begin{theorem}
    \label{thm:main}
    If a vertex transitive graph has at least one thick end, then it has infinite weak cop number.
\end{theorem}
\begin{proof}
    We give a winning strategy for the robber against any finite number $k$ of cops. The cop player decides on the cop speed $s_c$ and the capture radius $\rho$ before the robber has to make any decision, so we may assume that these parameters are fixed throughout the rest of the proof. Fix a root vertex $v_0 \in V(G)$ and a thick end $\omega$ of $G$. Since all spheres and balls in the remainder of the proof are centred at $v_0$, let us define $B(n) = B(n,v_0)$ and $S(n) = S(n,v_0)$. 

    In order to choose the robber speed $s_r$, we first recursively define a sequence $R_i$ as follows. Choose $R_0$ such that there are $k \cdot b(s_c+\rho)+1$ vertex disjoint rays in $\omega$ starting in $B(R_0)$. Note that this is possible since $\omega$ is a thick end. Assuming that $R_i$ is already defined, let $X_i$ be the set of vertices in $S(R_i+1)$ that lie on some ray in $\omega$ which does not meet $B(R_i)$. Since any two rays in $\omega$ are connected by infinitely many disjoint paths, any two vertices in $X_i$ are connected by a path in $G\setminus B(R_i)$. Since $X_i$ is finite, it is possible to choose $R_{i+1}$ so that  any two vertices in $X_i$ are connected by a path in $ B(R_{i+1})\setminus B(R_i)$. Finally, we set $s_r = b(R_{k \cdot b(\rho)+1})$.

    Call a vertex \emph{open} if its distance from any cop is larger than $\rho$, and \emph{safe} if its distance from any cop is larger than $s_c+\rho$. Note that every safe vertex is open and remains open even after the cops' next move. Call a vertex in $B(R_0)$ a \emph{haven} if it is the starting vertex of a ray in $\omega$ all of whose vertices are safe. 
    
    Note that havens exist: indeed, by the definition of $R_0$, there are more than $k \cdot b(s_c+\rho)$ vertex disjoint rays in $\omega$ starting in $B(R_0)$. Since each cop makes only $b(s_c+\rho)$ vertices unsafe, at least one of those rays consists entirely of safe vertices. An analogous argument shows that at least one of the sets $B(R_{i}) \setminus B(R_{i-1})$ for $i \leq k \cdot b(\rho)+1$ consists entirely of open vertices.
    
    The robber's strategy is to always move to a haven; if this is possible, then it clearly allows the robber to visit $B(R_0)$ infinitely many times, thereby winning the game for $R = R_0$. Since havens exist regardless of the cops' positions, the robber may choose to start at a haven. 
    
    Assume now that the robber is at a haven $v$ after the $t$-th robber move, and let $P$ be a ray in $\omega$ starting at $v$ and containing only safe vertices. After the cops' $(t+1)$-th move,  all vertices of $P$ are still open, and  there is a new haven $w$. Let $Q$ be a ray in $\omega$ starting at $w$ and consisting entirely of safe (and hence open) vertices. Choose $i \leq k \cdot b(\rho)+1$ such that all vertices of $B(R_{i}) \setminus B(R_{i-1})$ are open. By definition of $X_{i-1}$, the rays $P$ and $Q$ contain vertices $p$ and $q$ in $X_{i-1}$, respectively. The path obtained by following $P$ from $v$ to $p$, then a path from $p$ to $q$ in $B(R_{i}) \setminus B(R_{i-1})$, and finally $Q$ from $q$ to $w$ (potentially removing cycles and backtracking subpaths) consists entirely of open vertices. Since this path is contained in $B(R_{i}) \subseteq B(R_{k \cdot b(\rho)+1})$, its length cannot be larger than $s_r$, and thus the robber can use this path to move to the haven $w$.
\end{proof}

\begin{remark}
    In the above proof, vertex transitivity is only used to bound the number of vertices a cop can prevent from being safe or open. The same proof still works for graphs containing a thick end where there is some bound $b(n)$ on the size of $n$-balls.
\end{remark}

\begin{corollary}
    If $G$ is a connected, vertex transitive graph, then $\wcop (G) \in \{1,\infty\}$.
\end{corollary}
\begin{proof}
    If at least one end of $G$ is thick, then $\wcop(G)=\infty$ by Theorem \ref{thm:main}.
    If all ends of $G$ are thin, then $G$ is quasi-isometric to a tree (by \cite{zbMATH04139775}), and therefore $\wcop(G)=1$ (by \cite{appenzeller2025copsrobbershyperbolicvirtually} or \cite{esperet2025coarsecopsrobbergraphs}, though only the `easy' direction that quasi-isometric to a tree implies $\wcop(G)=1$ is needed). 
\end{proof}

\bibliographystyle{abbrv}
\bibliography{sources}

\end{document}